\documentclass[11pt]{amsart}
\usepackage{amssymb}

\makeatletter
\newcommand{\sumprime}{\if@display\sideset{}{'}\sum%
            \else\sum'\fi}
\makeatother

\begin{document}

\numberwithin{equation}{section}

% define theorem environments
\newtheorem{theorem}{Theorem}[section]
\newtheorem{proposition}[theorem]{Proposition}
\newtheorem{conjecture}[theorem]{Conjecture}
\def\theconjecture{\unskip}
\newtheorem{corollary}[theorem]{Corollary}
\newtheorem{lemma}[theorem]{Lemma}
\newtheorem{observation}[theorem]{Observation}
\newtheorem{definition}{Definition}
\numberwithin{definition}{section} %\def\thedefinition{\unskip}
\newtheorem{remark}{Remark}
\def\theremark{\unskip}
\newtheorem{question}{Question}
\def\thequestion{\unskip}
\newtheorem{example}{Example}
\def\theexample{\unskip}
\newtheorem{problem}{Problem}

\def\vvv{\ensuremath{\mid\!\mid\!\mid}}
\def\intprod{\mathbin{\lr54}}
\def\reals{{\mathbb R}}
\def\integers{{\mathbb Z}}
\def\N{{\mathbb N}}
\def\complex{{\mathbb C}\/}
\def\dist{\operatorname{dist}\,}
\def\spec{\operatorname{spec}\,}
\def\interior{\operatorname{int}\,}
\def\trace{\operatorname{tr}\,}
\def\cl{\operatorname{cl}\,}
\def\essspec{\operatorname{esspec}\,}
\def\range{\operatorname{\mathcal R}\,}
\def\kernel{\operatorname{\mathcal N}\,}
\def\dom{\operatorname{Dom}\,}
\def\linearspan{\operatorname{span}\,}
\def\lip{\operatorname{Lip}\,}
\def\sgn{\operatorname{sgn}\,}
\def\Z{ {\mathbb Z} }
\def\e{\varepsilon}
\def\p{\partial}
\def\rp{{ ^{-1} }}
\def\Re{\operatorname{Re\,} }
\def\Im{\operatorname{Im\,} }
\def\dbarb{\bar\partial_b}
\def\eps{\varepsilon}
\def\O{\Omega}
\def\Lip{\operatorname{Lip\,}}

\def\Hs{{\mathcal H}}
\def\E{{\mathcal E}}
\def\scriptu{{\mathcal U}}
\def\scriptr{{\mathcal R}}
\def\scripta{{\mathcal A}}
\def\scriptc{{\mathcal C}}
\def\scriptd{{\mathcal D}}
\def\scripti{{\mathcal I}}
\def\scriptk{{\mathcal K}}
\def\scripth{{\mathcal H}}
\def\scriptm{{\mathcal M}}
\def\scriptn{{\mathcal N}}
\def\scripte{{\mathcal E}}
\def\scriptt{{\mathcal T}}
\def\scriptr{{\mathcal R}}
\def\scripts{{\mathcal S}}
\def\scriptb{{\mathcal B}}
\def\scriptf{{\mathcal F}}
\def\scriptg{{\mathcal G}}
\def\scriptl{{\mathcal L}}
\def\scripto{{\mathfrak o}}
\def\scriptv{{\mathcal V}}
\def\frakg{{\mathfrak g}}
\def\frakG{{\mathfrak G}}

\def\ov{\overline}

\thanks{Supported by Grant IDH1411001 from Fudan University}

\address{School of Mathematical Sciences, Fudan University, Shanghai 200433, China}
 \email{boychen@fudan.edu.cn}

\title{Convergence of Riemannian surfaces and convergence of the Bergman kernel}
\author{Bo-Yong Chen}
\date{}
\maketitle

\bigskip

\begin{abstract}
  Let $\{M_j\}$ be a sequence of complete Riemannian surfaces which converges in the sense of Cheeger-Gromov to a complete Riemannian surface $M$. We study the convergence of the Bergman kernel $K_{M_j}$ of $M_j$ by using isoperimetric inequalities.

  \bigskip

\end{abstract}

\section{Introduction}

Let $M$ be an orientable surface, i.e., an orientable differentiable $2-$manifold. By means of patching up together local metrics through a partition of unity, we see that $M$ admits many Riemannian metrics. Let
$$
ds^2=E(x,y)dx^2+2F(x,y)dxdy+G(x,y)dy^2
$$
where $EG-F^2>0,\ E>0$, be a (smooth) Riemannian metric defined in local coordinates $(x,y)$ of $M$. By\/ {\it isothermal parameters}\/ we mean local coordinates $(\xi,\zeta)$ with $\xi=\xi(x,y),\,\zeta=\zeta(x,y)$, such that
$$
ds^2=\lambda(\xi,\zeta)(d\xi^2+d\zeta^2),\ \ \ \lambda(\xi,\zeta)>0.
$$
Such isothermal parameters are known to exist by the famous Korn-Lichtenstein theorem, which goes back to Gauss. With respect to local coordinates $z=\xi+\zeta i$, $M$ becomes a complex manifold. This observation is significant since the complex structure of a given surface is often unknown, whereas the Riemannian metric can be analyzed by means from Riemannian geometry.

In this paper, we attempt to understand stability properties of complex analytic objects for a sequence of Riemannian surfaces which converges in the following sense:

   \begin{definition}[cf. \cite{GromovMetric}, see also \cite{Hamilton}, \cite{Topping}]
   A sequence $\{(M_j,ds^2_j)\}$ of complete Riemannian manifolds is said to converge in the sense of Cheeger-Gromov to a complete Riemannian manifold $(M,ds^2)$ if there exist
   \begin{enumerate}
   \item a sequence of points $p_j\in M_j$ and a point $p\in M$;
   \item a sequence of precompact open sets $\Omega_j\subset M$ exhausting $M$, with $p\in \Omega_j$ for each $j$;
   \item a sequence of smooth maps $\phi_j:\Omega_j\rightarrow M_j$ which are diffeomorphic onto their image and satisfy $\phi_j(p)=p_j$;
   \end{enumerate}
   such that $\phi_j^\ast(ds^2_j)\rightarrow ds^2$ in the sense that for all compact subsets $E\subset M$, the tensor $\phi_j^\ast(ds^2_j)- ds^2$ and its covariant derivatives of all orders (with respect to any fixed background connection)  converge uniformly to zero on $E$.
  \end{definition}

  More precisely, we are interested in the following

  \begin{problem}
   Let  $\{(M_j,ds^2_j)\}$ be a sequence of complete Riemannian surfaces which converges in the sense of Cheeger-Gromov to a complete Riemannian surface $(M,ds^2)$. Let $K_{M_j}$ (resp. $K_M$) be the Bergman kernel of $M_j$ (resp. $M$), with respect to the corresponding complex structure. When does $K_{M_j}$ converge to $K_M$ in the sense that for all compact sets $E\subset M$, the tensors $\phi_j^\ast(K_{M_j})- K_M$ and their covariant derivatives of all orders (with respect to any fixed background connection)  converge uniformly to zero on $E$?
  \end{problem}

 Here the Bergman kernel is the reproducing kernel of the Hilbert space of square-integrable holomorphic differentials, which is a  classical conformal invariant. Since there are plenty of convergent sequence of Riemannian surfaces whose Bergman kernels do not converge (see \S10, Remark 1), so we have to find reasonable sufficient conditions.  A popular global property in geometric analysis is so-called isoperimetric inequalities which we recall as follows.
   Let $M$ be a noncompact complete Riemannian manifold. For each $1\le \nu\le \infty$, the $\nu-$dimensional isoperimetric constant $I_\nu(M)$ of $M$ is defined by
  $$
  I_\nu(M)=\inf |\partial \Omega|/|\Omega|^{1-1/\nu}
  $$
  where the infimum is taken over all precompact domains $\Omega\subset M$ with smooth boundaries, and $|\cdot|$ stands for the volume. If $I_\nu(M)>0$, then $M$ satisfies isoperimetric inequalities $|\partial \Omega|\ge I_\nu(M)|\Omega|^{1-1/\nu}$ for all $\Omega$.
  In case that $M$ is compact, we have to adjust the definition of $I_\nu(M)$ as follows
  $$
  I_\nu(M)=\inf \frac{|S|}{\min\{|\Omega_1|^{1-1/\nu},|\Omega_2|^{1-1/\nu}\}}
  $$
 where the infimum is taken over all compact smooth hypersurface $S$ of $M$ that divide $M$ into two disjoint open subsets $\Omega_1,\Omega_2$ of $M$. The number $I_\infty(M)$ is also called Cheeger's constant in the literature.

  Our main result is stated as follows.

   \begin{theorem}\label{th:convergence}
 Let $\{(M_j,ds^2_j)\}$ be a sequence of complete Riemannian surfaces which converges in the sense of Cheeger-Gromov to a\/ {\rm noncompact\/} complete Riemannian surface $(M,ds^2)$.  Suppose one of the following conditions is verified:
 \begin{enumerate}
 \item  $\inf_j I_\infty(M_j)>0$, where $M_j$ can be compact or noncompact.
 \item  $\inf_j I_\nu(M_j)>0$ for some $2<\nu<\infty$, where $M_j$ is noncompact.
 \item  $\inf_j {I_2(M_j)}{{|M_j|}^{-1/2}}>0$, where $M_j$ is compact.
 \end{enumerate}
            Then  $K_{M_j}$  converges to  $K_M$.
  \end{theorem}

   \begin{remark}
   The case when $M$ is compact is not very interesting since $M_j$ would be diffeomorphic onto $M$ for all sufficiently large $j$. Thus the classical theory on deformation of complex structures applies (compare \cite{KodairaBook}).
  \end{remark}

  The idea of using the length-area method goes back to Beurling and Ahlfors, which plays  an important role in the study of complex analysis on noncompact Riemannian surfaces (see e.g., \cite{AhlforsSario}, Chapter IV).

  Condition (1) or (3) of Theorem \ref{th:convergence} can be replaced by the weaker condition $\inf_j \lambda_1(M_j)>0$, where $\lambda_1$ is the infimum of the spectrum of the  Laplacian. However, condition (2) does not yield $\inf_j \lambda_1(M_j)>0$. Thus it is reasonable to consider the case when $\lambda_1(M_j)$ degenerates, even the convergence of Riemannian surfaces is quite special.

  \begin{theorem}\label{th:convergenceEffective}
  Let $\{(M_j,ds^2_j)\}$ be a sequence  of complete Riemannian surfaces and  $(M,ds^2)$ a complete Riemannian surface. Suppose that there exists a sequence of geodesic balls $B_{R_j}(p)$ in $M$ with $p$ fixed and $R_j\rightarrow \infty$, such that
     \begin{enumerate}
     \item $B_{R_j}(p)\subset M_j$ for all $j$;
   \item $ds^2_j=ds^2$ on $B_{R_j}(p)$;
   \item $
   \inf_j \lambda_1(M_j)R_j^2>0.
   $
   \end{enumerate}
   Then $K_{M_j}$ converges to $K_M$.
   \end{theorem}

When $M$ is a ${\mathbb Z}$ covering of a compact Riemannian surface with genus $\ge 2$, we may construct a sequence $\{M_j\}$ of compact Riemannian surfaces  which converges to $M$ as Theorem \ref{th:convergenceEffective}, whereas Theorem \ref{th:convergence} does not apply (see \S10, Remark 6).

       The paper is organized as follows. In \S2 we recall necessary background from  geometric analysis. In \S3 we review basic properties of isoperimetric inequalities. In \S4 we estimate the Green function by using isoperimetric inequalities.  Sections 5,6,7,8 are devoted to the proof of Theorem \ref{th:convergence}. In \S9 we prove Theorem \ref{th:convergenceEffective}. In \S10 we present a number of remarks.

\section{Basic facts from  geometric analysis}

   Let $ds^2=g_{ij}dx^idx^j$ be a Riemannian metric on $M$. The Laplace operator is defined by
  $$
  \Delta=g^{-1/2}\frac{\partial}{\partial x^i}\left(g^{1/2} g^{ij}\frac{\partial}{\partial x^j}\right)
  $$
  where $(g^{ij})=(g_{ij})^{-1}$ and $g={\rm det}(g_{ij})$. The gradient $\nabla$ acts on a function $u$ by
  $$
  (\nabla u)^i=g^{ij}\frac{\partial u}{\partial x^j}.
  $$
  Green's formula asserts that for any  precompact domain $\Omega\subset M$ with a $C^1-$smooth boundary,
  $$
  \int_\Omega v\Delta u=\int_{\partial \Omega} v\frac{\partial u}{\partial \vec{n}}-\int_\Omega \nabla v\nabla u
  $$
   for all $u,v\in C^2(\Omega)\cap C^1(\overline{\Omega})$, where $\vec{n}$ denotes the outward unit normal vector field on $\partial \Omega$.

   The heat kernel $p(t,x,y)$ of $M$ is  the smallest positive fundamental solution to the heat equation
   $$
   \partial u/\partial t=\Delta u.
   $$
   More precisely, it is given by
   $$
   p(t,x,y)=\lim_{j\rightarrow \infty} p_j(t,x,y)
   $$
   where $p_j(t,x,y)$ is the Green function for the Dirichlet problem for the heat equation on the precompact open subset $\Omega_j$, $j=1,2,\cdots,$ which exhausts $M$ (see e.g. \cite{GrigoryanHeat}). Some basic properties are as follows.
    \begin{enumerate}
   \item $p(t,x,y)=p(t,y,x)$.
    \item $p(t,\cdot,y)\rightarrow \delta_y$ as $t\rightarrow 0+$, where $\delta_y$ denotes the Dirac distribution.
    \item The semigroup property: for all $t,s>0$ and $x,y\in M$,
    $$
    p(t+s,x,y)=\int_M p(t,x,z)p(s,y,z)dV_z.
    $$
    \item $\int_M p(t,x,y) dV_y\le 1$.
    \end{enumerate}

      A positive increasing function $\kappa$ on $(0,\infty)$ is called\/ {\it regular\/} if there are numbers $A\ge 1$ and $\beta>1$ such that
   \begin{equation}\label{eq:regular}
   \frac{\kappa(\beta s)}{\kappa(s)}\le A \frac{\kappa(\beta t)}{\kappa(t)},\ \ \ {\forall\ } 0<s<t.
   \end{equation}
    Grigor'yan made a deep observation  that\/ {\it off-diagonal\/} estimates of the heat kernel may be deduced from (easier)\/ {\it on-diagonal\/} estimates as follows.

   \begin{theorem}[cf. \cite{GrigoryanGaussian}]\label{th:heatGaussianBound}
   Let $x,y$ be two points in $M$ such that for all $t>0$
   $$
   p(t,x,x)\le 1/\kappa_1(t),\ \ \ p(t,y,y)\le 1/\kappa_2(t)
   $$
   where $\kappa_1,\kappa_2$ are two regular functions. Then for any number $\alpha<1$ and all $t>0$
   \begin{equation}\label{eq:heatOff-diagonal}
   p(t,x,y)\le \frac{4A}{\sqrt{\kappa_1(\delta t)\kappa_2(\delta t)}}\exp\left(-\frac{\alpha d^2(x,y)}{4 t}\right)
   \end{equation}
   where $\delta=\delta(\beta,\alpha)$ and $A,\beta$ are the constants from (\ref{eq:regular}).
   \end{theorem}

   Finally, let $M$ be a Riemannian $n-$manifold with Ricci curvature  $\ge -(n-1)b^2$ where $b\ge 0$. Let $B_r(x)$ denote the geodesic ball with center $x$ and radius $r$. Suppose $\overline{B_r(x)}\subset M$. Then we have
   \begin{enumerate}
   \item Harnack's inequality (cf. \cite{SchoenYauBook}): For any positive harmonic function $u$ on $B_r(x)$,
    \begin{equation}\label{eq:Harnack}
  \sup_{B_{r/2}(x)} u\le e^{{\rm const}_n \left(1+br\right)} \inf_{B_{r/2}(x)} u.
    \end{equation}
        \item The sub-mean-value inequality (cf. \cite{LiSchoen}): For any  positive subharmonic function $u$ on $B_r(x)$,
    \begin{equation}\label{eq:meanvalue}
     \sup_{B_{r/2}(x)} u^2 \le e^{{\rm const}_n(1+br)}|B_r(x)|^{-1}\int_{B_r(x)}u^2.
    \end{equation}
      \end{enumerate}

 For further knowledge on geometric analysis, one may consult the book of Schoen-Yau \cite{SchoenYauBook} and survey articles of Grigor'yan \cite{GrigoryanHeat}, \cite{Grigoryan}.

   \section{Isoperimetric inequalities}

  We follow closely the books of Chavel \cite{ChavelIsoper}, \cite{ChavelRiemann}. Let $M$ be a noncompact complete Riemannian $n-$manifold. Let ${\mathcal F}$ denote the set of precompact domains in $M$ with smooth boundaries. For each $1\le \nu\le \infty$, the $\nu-$dimensional isoperimetric constant $I_\nu(M)$ of $M$ is defined by
  $$
  I_\nu(M)=\inf_{\Omega\in {\mathcal F}} |\partial \Omega|/|\Omega|^{1-1/\nu}.
  $$
  Similarly, we may define for each $\nu>1$ the $\nu-$dimensional Sobolev constant
  \begin{equation}\label{eq:Sobolev}
  S_\nu(M)=\inf \left\{ \|\nabla u\|_{1}/\|u\|_{\nu/(\nu-1)}: u\in C_0^\infty(M)\right\}
  \end{equation}
  where $C^\infty_0(M)$ denotes the set of smooth functions with compact supports in $M$ and $\|\cdot\|_p$ stands for the standard $L^p-$norm.
  The famous Federer-Fleming-Maz'ya inequality yields
  \begin{equation}\label{eq:Federer}
  I_\nu(M)=S_\nu(M)
  \end{equation}
   for all $\nu\in (1,\infty]$.
   For each $\phi\in C^\infty_0(M)$, we put $u=|\phi|^{2(\nu-1)/(\nu-2)}$ for some $\nu>2$. By (\ref{eq:Sobolev}), (\ref{eq:Federer}) and the Cauchy-Schwarz inequality, we immediately get the following $L^2$ Sobolev inequality
  \begin{equation}\label{eq:L2Sobolev}
  \|\nabla \phi\|_2\ge \frac{\nu-2}{2(\nu-1)}I_\nu(M)\|\phi\|_{2\nu/(\nu-2)}.
  \end{equation}
 Thanks to the H\"older inequality, we have
 $$
 \int_M \phi^{2+4/\nu}=\int_M \phi^2\cdot\phi^{4/\nu}\le \left(\int_M \phi^{2\nu/(\nu-2)}\right)^{(\nu-2)/\nu} \left(\int_M \phi^2\right)^{2/\nu}
 $$
 and
$$
 \int_M \phi^2=\int_M \phi^{4/(\nu+4)}\cdot\phi^{(2\nu+4)/(\nu+4)}\le \left(\int_M |\phi|\right)^{4/(\nu+4)} \left(\int_M \phi^{2+4/\nu}\right)^{\nu/(\nu+4)}.
 $$
 Together with (\ref{eq:L2Sobolev}), we obtain Nash's inequality
 \begin{equation}\label{eq:NashIneq}
 \|\phi\|_2^{2+4/\nu}\le \left(\frac{\nu-2}{2(\nu-1)}I_\nu(M)\right)^{-2} \|\nabla \phi\|_2^2\cdot\|\phi\|_1^{4/\nu},\ \ \ \forall\,\nu>2.
 \end{equation}

    A central property of $I_\nu$ is that it behaves well under\/ {\it rough isometries}.  Following Kanai \cite{KanaiRough}, we call a map $\Phi:M_1\rightarrow M_2$ between two Riemannian manifolds $M_1$ and $M_2$ a\/ {\it rough isometry\/} if there are constants $a\ge 1$ and $b\ge 0$ such that
  $$
  a^{-1} d_1(x,y)-b\le d_2(\Phi(x),\Phi(y))\le a d_1(x,y)+b
  $$
  for all $x,y\in M_1$, and $\Phi$ is $r-${\it full} for some $r>0$, i.e.,
  $$
  \bigcup_{x\in M_1}B_r(\Phi(x))=M_2.
  $$
  A complete Riemannian manifold $M$ has\/ {\it bounded geometry} if the Ricci curvature is bounded below by a constant, and the injectivity radius ${\rm inj}(M)$ of $M$ is positive.

    \begin{theorem}[cf. \cite{KanaiRough}]\label{th:Kanai}
  Let $M_1,M_2$ be complete Riemannian manifolds with bounded geometries such that they are roughly isometric to each other. Let
  $$
  \nu\ge \max\{{\rm dim\,}M_1,{\rm dim\,}M_2\}.
  $$
   Then $I_\nu(M_1)>0$ if and only if $I_\nu(M_2)>0$.
  \end{theorem}

  Below we collect some examples concerning positive isoperimetric constants (cf. \cite{GrigoryanHeat}, \S\,7\,):

  \begin{enumerate}
  \item Any Cartan-Hadamard $n-$manifold has $I_n(M)>0$.
  \item Any Cartan-Hadamard manifold $M$ with sectional curvature $\le -b^2$ ($b>0$) has $I_\infty(M)>0$.
  \item Any $n-$dimensional minimal submanifold $M$ in ${\mathbb R}^N$ has $I_n(M)>0$. Note that any complex submanifold in ${\mathbb C}^n$ is minimal.
     \end{enumerate}

    A useful isoperimetric inequality is given by Coulhon and Saloff-Coste (cf. \cite{CoulhonSaloff}, Theorem 4, see also \cite{Grigoryan}, Theorem 11.3) as follows. Let $M$ be a noncompact regular covering of a compact manifold $M_0$. Put
  $$
  V(r):=|B_r(x_0)|,
  $$
  where $x_0$ is some (fixed) point in $M$. For some (large) constant $C>0$, the isoperimetric inequality
  \begin{equation}\label{eq:Coulhon}
  |\partial \Omega|\ge \frac{|\Omega|}{CV^{-1}(C|\Omega|)}
  \end{equation}
  holds for all precompact domains $\Omega\subset M$ with smooth boundaries and $|\Omega|\ge {\rm const.}>0$. Here $V^{-1}$ is the inverse function of $V$. In particular, any ${\mathbb Z}^m$ ($m\ge 2$) covering of a compact Riemannian manifold has $I_m>0$.

    There is also a beautiful example from\/ {\it hyperbolic\/} geometry. Let ${\mathbb H}^{n}$ be the hyperbolic space. A hyperbolic manifold is given by $M={\mathbb H}^n/\Gamma$ where $\Gamma$ is a free, discrete group of hyperbolic isometries. The\/ {\it critical exponent\/} $\delta(\Gamma)$ of Poincar\'e series is defined by
    $$
    \delta(\Gamma)=\inf\left\{s\ge 0: \sum_{\sigma\in \Gamma} \exp({-sd(x,\sigma(y))})<\infty\right\}
    $$
    for some/any $x,y\in {\mathbb H}^n$, where $d$ denotes the hyperbolic distance. It is well-known that $\delta(\Gamma)\le n-1$. Let $\lambda_1(M)$ be the fundamental tone of $M$, i.e., the infimum of the spectrum of $-\Delta$. The quantities $I_\infty(M)$, $\lambda_1(M)$ and $\delta(\Gamma)$ are related through the following
       \begin{enumerate}
      \item Cheeger's inequality (cf. \cite{CheegerEigenvalue}): $\lambda_1(M)\ge I_\infty(M)^2/4$ (actually holds arbitrary complete manifolds);
      \item Sullivan's theorem (cf. \cite{SullivanEigenvalue}): $\lambda_1(M)=(n-1)^2/4$ if $\delta(\Gamma)\le (n-1)/2$, and $\lambda_1(M)=\delta(\Gamma)(n-1-\delta(\Gamma))$ otherwise;
      \item Buser's inequality (cf. \cite{BuserIsoper}): $\lambda_1(M)\le {\rm const}_n\,I_\infty(M)$ (actually holds for arbitrary noncompact complete manifolds with Ricci curvature $\ge -1$).
      \end{enumerate}
    It follows immediately that
    \begin{equation}\label{eq:PoincareExp}
    \delta(\Gamma)<n-1 \iff
    I_\infty(M)>0.
    \end{equation}
        In particular, most hyperbolic Riemannian surfaces have $I_\infty(M)>0$.

 Based on Theorem \ref{th:Kanai} and the examples above, we may construct many complete Riemannian\/ {\it surfaces\/} with $I_\nu>0$ for some $\nu>2$. For instance, a $2-$dimensional jungle gym in ${\mathbb R}^n$ ($n>2$) has $I_n>0$, whereas a $2-$dimensional jungle gym in a Cartan-Hadamard manifold with sectional curvature $\le -b^2$ ($b>0$) has $I_\infty>0$.

 \section{Estimates of the Green function}

    Let $M$ be a noncompact complete Riemannian $n-$manifold. Let $\Omega\subset M$ be an open set and $U$ be a precompact open set in $\Omega$. The capacity ${\rm cap}(U,\Omega)$ is defined as follows
     $$
     {\rm cap}(U,\Omega)=\inf\int_\Omega |\nabla \phi|^2
     $$
     where the infimum is taken over all locally Lipschitz functions on $M$ with compact supports in $\overline{\Omega}$ such that $0\le \phi\le 1$ and $\phi|_{\overline{U}}=1$.
     Let $g_\Omega$ denote the (positive) Green function of $\Omega$. There is a useful link between the capacity and the Green function as follows
     \begin{equation}\label{eq:GreenVsCapacity}
     \inf_{\partial U} g_\Omega(\cdot,y)\le {\rm cap\,}(U,\Omega)^{-1}\le  \sup_{\partial U} g_\Omega(\cdot,y),\ \ \ {\forall\,} y\in U
     \end{equation}
     (cf. \cite{Grigoryan}, Proposition 4.1).

    Now fix $o\in M$ and let $B_R:=B_R(o)$ with $R>1$.  Put
    $$
    \varepsilon_R:=\min\left\{\inf_{x\in B_{R+1}}{\rm inj}(M,x),1/2\right\},
    $$
    where ${\rm inj}(M,x)$ denotes the injectivity radius at $x$. We give first a rough lower bound for the Green function $g_M$ of $M$ as follows.

    \begin{proposition}\label{prop:GreenLowerBound}
    Suppose ${\rm Ricci}(M)\ge -(n-1)b^2$ $(b\ge 0)$ on $B_{R+1}$. Then
    \begin{equation}\label{eq:GreenLowerBound}
     g_M(x,o)\ge \frac18\, |B_{R+1}|^{-1}\exp\{-{\rm const}_n \left(1+b\varepsilon_R\right)\varepsilon_R^{-n}|B_{R+1}|\}.
    \end{equation}
   for all $x\in B_R$.
    \end{proposition}

    \begin{proof}
    Take $\{x_1,\cdots,x_m\}\subset \partial B_R$ such that $B_{\frac12{\varepsilon_R}}(x_1),\cdots,B_{\frac12\varepsilon_R}(x_m)$ do not overlap and $B_{\varepsilon_R}(x_1),\cdots,B_{\varepsilon_R}(x_m)$ cover $\partial B_R$. By a well-known theorem of Croke \cite{CrokeIsoperimetric} we have
    $$
    \left|B_{\frac12{\varepsilon_R}}(x_k)\right|\ge {\rm const}_n\, \varepsilon_R^n
    $$
    for all $k$. Thus
    $$
    |B_{R+1}|\ge \sum_k  \left|B_{\frac12{\varepsilon_R}}(x_k)\right|\ge m\,{\rm const}_n\, \varepsilon_R^n,
    $$
    i.e.,
    $$
    m\le ({\rm const}_n\, \varepsilon_R^n)^{-1} |B_{R+1}|.
    $$
        Since $g_{B_{R+1}}(\cdot,o)$ is harmonic on $B_{R+1}\backslash \{o\}$, it follows from Harnack's inequality (\ref{eq:Harnack}) that
    $$
     \sup_{B_{\varepsilon_R}(x_k)} g_{B_{R+1}}(\cdot,o)\le e^{{\rm const}_n \left(1+b\varepsilon_R\right)} \inf_{B_{\varepsilon_R}(x_k)} g_{B_{R+1}}(\cdot,o)
    $$
    for all $k$. Since $B_{\varepsilon_R}(x_1),\cdots,B_{\varepsilon_R}(x_m)$ cover $\partial B_R$, so we have
        \begin{eqnarray}\label{eq:GreenHarnack}
     \sup_{\partial B_R} g_{B_{R+1}}(\cdot,o) & \le & e^{{\rm const}_n \left(1+b\varepsilon_R\right)m} \inf_{\partial B_R} g_{B_{R+1}}(\cdot,o)\nonumber\\
     & \le & e^{{\rm const}_n \left(1+b\varepsilon_R\right)\varepsilon_R^{-n}|B_{R+1}|} \inf_{\partial B_R} g_{B_{R+1}}(\cdot,o).
    \end{eqnarray}
    By virtue of Theorem 7.1 in \cite{Grigoryan}, we have
    \begin{eqnarray*}
    {\rm cap}(B_R,B_{R+1})^{-1} & \ge & \frac12 \int_R^{R+1}\frac{(t-R)dt}{|B_t|-|B_R|}\\
    & \ge & \frac12 \int_{R+1/2}^{R+1}\frac{(t-R)dt}{|B_t|-|B_R|}\\
    & \ge & \frac18\, |B_{R+1}|^{-1}.
    \end{eqnarray*}
    Together with (\ref{eq:GreenVsCapacity}) and (\ref{eq:GreenHarnack}), we get
    $$
    \inf_{\partial B_R} g_{M}(\cdot,o)\ge \inf_{\partial B_R} g_{B_{R+1}}(\cdot,o)\ge  \frac18\, |B_{R+1}|^{-1}\exp\{-{\rm const}_n \left(1+b\varepsilon_R\right)\varepsilon_R^{-n}|B_{R+1}|\}.
    $$
    The assertion follows immediately from the maximal principle.
    \end{proof}

     In what follows in this section we always assume that $M$ is a noncompact complete Riemannian manifold with $I_\nu(M)>0$ for some $2<\nu<\infty$. We have the following (probably optimal) upper bound for the Green function.

  \begin{proposition}[cf. \cite{ChavelFeldman}]\label{prop:Green2}
   We have
   \begin{equation}\label{eq:GreenEstimate}
   g_M(x,y)\le {\rm const}_\nu\,I_\nu(M)^{\nu}  d(x,y)^{2-\nu}
   \end{equation}
   for all $x,y\in M$.
  \end{proposition}

   In order to make the paper self-contained, we include the proof here. The key point is to obtain the following Gaussian upper bound for the heat kernel.

  \begin{theorem}[cf. \cite{ChavelFeldman} or \cite{GrigoryanHeat}]\label{th:heatIsoper}
  For any $\alpha<1$,
  \begin{equation}\label{eq:heatIsoper}
  p(t,x,y)\le {\rm const}_{\nu,\alpha}\,I_\nu(M)^{\nu} t^{-\nu/2}\exp\left(-\frac{\alpha d^2(x,y)}{ 4t}\right)
  \end{equation}
   for all $x,y\in M$ and $t>0$.
  \end{theorem}

  Let us first observe how to derive (\ref{eq:GreenEstimate}) from (\ref{eq:heatIsoper}). Indeed,
  \begin{eqnarray*}
   g_M(x,y) & = &  \int_0^\infty p(t,x,y) dt\le {\rm const}_{\nu}\,I_\nu(M)^{\nu} \int_0^\infty t^{-\nu/2} \exp\left(-\frac{d^2(x,y)}{5t}\right)dt\\
    & \le & {\rm const}_{\nu}\,I_\nu(M)^{\nu} d(x,y)^{2-\nu}  \int_0^\infty t^{-\nu/2} \exp\left(-\frac{1}{5t}\right)dt.
  \end{eqnarray*}

    In order to verify (\ref{eq:heatIsoper}), we need an\/ {\it on-diagonal\/} upper bound for the heat kernel which goes back to Nash (see \cite{GrigoryanHeat}, \S\,6.1). By a standard exhaustion argument, it suffices to work on a precompact open set $\Omega\subset M$ with a smooth boundary. Fix $y\in \Omega$ and put $u(t,x):=p_\Omega(t,x,y)$ and
  $$
  J(t)=\int_\Omega u^2(t,x)dV_x.
  $$
  Note that
  \begin{eqnarray*}\label{eq:J-derivative}
  J'(t) & = & 2\int_\Omega uu_t=2\int_\Omega u\Delta u=-2\int_\Omega |\nabla u|^2\\
        & \le & -{\rm const}_\nu\, I_\nu(M)^{-2} \|u\|_2^{2+4/\nu}\|u\|_1^{-4/\nu}
  \end{eqnarray*}
  in view of Nash's inequality (\ref{eq:NashIneq}). Since
  $
  \|u\|_1=\|p_\Omega(t,\cdot,y)\|_1\le 1,
  $
  we have
  $$
  J'\le -{\rm const}_\nu\, I_\nu(M)^{-2} J^{1+2/\nu},
  $$
  so that
  $$
  J(t)\le {\rm const}_\nu\,I_\nu(M)^{\nu}t^{-\nu/2},
  $$
  for $J(0+)=\infty$. By the semigroup property, we obtain
  $$
  p_\Omega(t,y,y)=J(t/2)\le {\rm const}_\nu\,I_\nu(M)^{\nu} t^{-\nu/2}=:1/\kappa_\nu(t).
  $$
  For all $0<s<t$, we have
  $$
  \frac{\kappa_\nu(2s)}{\kappa_\nu(s)}=\frac{\kappa_\nu(2t)}{\kappa_\nu(t)},
  $$
  so that $\kappa_\nu$ is regular with $A=1$ and $\beta=2$. It follows from Theorem \ref{th:heatGaussianBound} that for any number $\alpha<1$,
  $$
  p(t,x,y)\le \frac{4}{\kappa_\nu(\delta_\alpha t)}\exp\left(-\frac{\alpha d^2(x,y)}{4t}\right)
  $$
  holds for suitable constant $\delta_\alpha>0$, from which inequality (\ref{eq:heatIsoper}) immediately follows.

   \begin{remark}
  \begin{enumerate}
    \item There are no analogous upper bounds for $g_M$ when $I_\infty(M)>0$ (consider the punctured disc with the Poincar\'e metric).
    \item It is interesting to note that Blocki used the classical isoperimetric inequality in ${\mathbb C}$ (i.e., $I_2({\mathbb C})>0$) to show that
    $$
    \log |\{g_\Omega(\cdot,y)>t\}|+2t
    $$
    is non-increasing in $t\in [0,\infty)$ for any $y\in \Omega \subset {\mathbb C}$ (see e.g., \cite{BlockCRvsMA}, Theorem 10.1).
   \end{enumerate}
   \end{remark}

       \section{Effective convergence of the Bergman kernel}

    Let $(M,ds^2)$ be a noncompact complete Riemannian  surface. Let $\lambda_1(M)$ be the infimum of the spectrum of $-\Delta$, i.e.,
    $$
    \lambda_1(M)=\inf \left\{\frac{\int_M |d f|^2}{\int_M |f|^2}:f\in C_0^\infty(M)\backslash \{0\}\right\}.
    $$
   Now view $(M,ds^2)$ as a complex manifold with $ds^2$ given by
    $
    \lambda  dw\,d\bar{w}
    $
    in local holomorphic coordinates. The\/ {\it complex\/} Laplace operator is defined by
    $$
    \Box=\lambda^{-1}\frac{\partial^2}{\partial w \partial\bar{w}}=\frac14 \Delta.
    $$
    Let $C_0^\infty(M,{\mathbb C})$ denote the set of complex-valued smooth functions on $M$ with compact supports. We begin with an elementary remark.

    \begin{lemma}\label{lm:eigenvalue}
        \begin{equation}\label{eq:eigenvalue}
    \lambda_1(M)=4 \inf \left\{\frac{\int_M |\partial f|^2}{\int_M |f|^2}:f\in C_0^\infty(M,{\mathbb C})\backslash \{0\}\right\}.
    \end{equation}
    \end{lemma}

    \begin{proof}
     For all $f\in C_0^\infty(M,{\mathbb C})$, we have
    $$
    \int_M |\bar{\partial}f|^2=\frac{\sqrt{-1}}2\int_M \partial \bar{f}\wedge \bar{\partial} f  =- \frac{\sqrt{-1}}2\int_M \bar{f} \partial\bar{\partial} f=\frac{\sqrt{-1}}2\int_M \partial {f}\wedge \bar{\partial} \bar{f}=\int_M |\partial f|^2,
    $$
    so that
    $$
    \lambda_1(M)\int_M |f|^2\le \int_M |df|^2 \le 2\int_M |\partial f|^2+2 \int_M |\bar{\partial} f|^2=4\int_M |\partial f|^2.
    $$
    On the other side, we may choose a sequence $f_j\in C_0^\infty(M,{\mathbb R})\backslash \{0\}$ such that
    $$
    \frac{\int_M |d f_j|^2}{\int_M |f_j|^2}\rightarrow \lambda_1(M).
    $$
    Since $f_j$ is real-valued, so we have $|df_j|^2=4|\partial f_j|^2$ and
    $$
    4\frac{\int_M |\partial f_j|^2}{\int_M |f_j|^2}\rightarrow \lambda_1(M).
    $$
        \end{proof}

        \begin{proposition}\label{th:L2dbar}
    Suppose $\lambda_1(M)>0$. Then for each $(1,1)-$form $v$ with $\|v\|_2<\infty$, there exists a solution $u$ of the equation $\bar{\partial}u=v$ such that
    $$
    \|u\|_2\le \frac2{\sqrt{\lambda_1(M)}}\|v\|_2.
    $$
    \end{proposition}

    \begin{proof}
    Let $D_{(p,q)}(M)$ denote the set of smooth $(p,q)-$forms with compact supports in $M$.  We introduce the following  inner product
    $$
    (f_1,f_2)=\int_M \phi_1 \bar{\phi}_2 \lambda^{-1} dV_w
    $$
    for all $f_1=\phi_1 dw\wedge d\bar{w},\,f_2=\phi_2 dw\wedge d\bar{w}\in D_{(1,1)}(M)$, where $dV_w=\frac{\sqrt{-1}}2 dw\wedge d\bar{w}$. Let $\bar{\partial}^\ast$ be the formal adjoint of $\bar{\partial}$. For any $u=\psi dw\in D_{(1,0)}(M)$ and $f=\phi dw\wedge d\bar{w}\in D_{(1,1)}(M)$, we have
    $$
    (f,\bar{\partial}u)=-\int_M \phi \frac{\partial \bar{\psi}}{\partial w} \lambda^{-1} dV_w=\int_M \frac{\partial}{\partial w}(\lambda^{-1}\phi)\bar{\psi} dV_w,
    $$
    so that
    $$
    \bar{\partial}^\ast f=\frac{\partial}{\partial w}(\lambda^{-1}\phi) dw.
    $$
    Since $\tilde{f}:=\lambda^{-1}\phi\in C_0^\infty(M,{\mathbb C})$, it follows that $\bar{\partial}^\ast f=\partial \tilde{f}$ and
    \begin{equation}\label{eq:PoincareIneq}
    \int_M |f|^2 \le \frac4{\lambda_1(M)}\int_M |\bar{\partial}^\ast f|^2
    \end{equation}
    in view of (\ref{eq:eigenvalue}). The remaining argument is standard (see e.g., \cite{HormanderConvexity}, p. 249). Given $v\in L^2_{(1,1)}(M,{\mathbb C})$,  the linear functional
    $$
    {\rm Range\,}\bar{\partial}^\ast \rightarrow {\mathbb C},\ \ \ \ \ \bar{\partial}^\ast f \mapsto (f,v)
    $$
    is well-defined and  bounded by $2\lambda_1(M)^{-1/2}\|v\|_2$ in view of (\ref{eq:PoincareIneq}). Thus by Hahn-Banach's theorem and the Riesz representation theorem, there is a unique $u\in L^2_{(1,0)}(M,{\mathbb C})$ such that
    $$
    (\bar{\partial}^\ast f,u)=(f,v)
    $$
    for all $f\in D_{(1,1)}(M)$, i.e., $\bar{\partial}u=v$ holds in the sense of distributions, such that
    $$
    \int_M |u|^2 \le \frac4{\lambda_1(M)}\int_M |v|^2.
    $$
    \end{proof}

    Let ${\mathcal H}(M)$ denote the Hilbert space of holomorphic differentials $h$ on $M$ satisfying
 $$
 \|h\|^2_2:=\frac{\sqrt{-1}}2 \int_M h\wedge \bar{h}<\infty.
 $$
  Let $\{h_j\}_{j=1}^\infty$ be a complete orthonormal basis of ${\mathcal H}(M)$. The Bergman kernel $K_M$ of $M$ is given by
   $$
   K_M(z)=\sum_j h_j(z)\wedge \overline{h_j(z)},\ \ \ {\forall\, }z\in M.
   $$

        \begin{proposition}\label{lm:globalApprox}
    Let $M$ be as the proposition above. Let $\rho$ denote the distance from some fixed point $o\in M$. Let $K_M$ and $K_{B_R}$ denote the Bergman kernel of $M$ and the geodesic ball $B_R:=B_R(o)$
        respectively. For each compact set $E\subset M$ with $o\in E$, we have
    \begin{equation}\label{eq:golbalApprox}
    \sup_{z\in E} \left|K_M(z)-K_{B_R}(z)\right|\le {\rm const.} \left( \lambda_1(M)^{-1/2} R^{-1}+\lambda_1(M)^{-1} R^{-2}\right)
    \end{equation}
    for all $R>2({\rm diam\,}E+1)$, where the constant depends only on $\inf_{x\in E} |B_1(x)|$ and the infimum of the Gaussian curvature of $ds^2$ on $E_1:=\{z\in M:{\rm dist\,}(z,E)\le 1\}$.
    \end{proposition}

    \begin{proof}
    Let $\chi:{\mathbb R}\rightarrow [0,1]$ be a smooth function satisfying $\chi|_{(-\infty,1/2)}=1$ and $\chi|_{(1,\infty)}=0$. Fix $R>2({\rm diam\,}E+1)$ for a moment. Let $h$ be a holomorphic differential on $B_R$ with $\|h\|_2\le 1$. Put
    $$
    v=\bar{\partial}\chi(\rho/R)\wedge h.
    $$
    By virtue of Proposition \ref{th:L2dbar}, there is a solution of $\bar{\partial}u=v$ on $M$ which satisfies
    \begin{equation}\label{eq:L2-u_2}
    \|u\|_2^2\le \frac4{\lambda_1(M)}\|v\|_2^2\le \frac4{\lambda_1(M)}\frac{\sup|\chi'|^2}{R^2}.
    \end{equation}
     Since $E_1\subset B_{R/2}$, we see that $u$ is\/ {\it holomorphic\/} on $E_1$, for $v=0$ holds on $B_{R/2}$. Write
     $$
     u=\phi dw\ \ \  {\rm and\ \ \ }
     ds^2=\lambda dw\,d\bar{w}
     $$
     in local holomorphic coordinates. Then we have
     $$
     \log |u|^2=\log |\phi|^2-\log \lambda,
     $$
     so that
     $$
     \Delta \log |u|^2=4\,\Box \log |u|^2\ge -\frac4{\lambda} \frac{\partial \log \lambda}{\partial w\partial \bar{w}}\ge -2 b^2
     $$
     holds on $E_1$, where $-b^2$ $(b\ge 0)$ denotes the infimum of the Gaussian curvature of $ds^2$ on $E_1$. It follows immediately that
     $$
     \hat{u}(t,z):=b^2 t^2 + \log |u|^2
     $$
      is subharmonic with respect to the metric $dt^2+ds^2$ on ${\mathbb R}\times E_1$, so is $e^{\hat{u}/2}$.  Applying the sub-mean-value inequality (\ref{eq:meanvalue}) to $e^{\hat{u}/2}$, we obtain that for any $z\in E$,
     \begin{eqnarray*}
     |u|^2(z)= e^{\hat{u}(0,z)} & \le & e^{C_0(1+b)}|B_1(0,z)|^{-1}\int_{B_1(0,z)} e^{\hat{u}}\\
     & \le & e^{C_0(1+b)}\int_0^1 e^{b^2t^2}dt \cdot |B_1(z)|^{-1}\int_{B_1(z)} |u|^2\\
     & \le & e^{C_0(1+b)}\int_0^1 e^{b^2t^2}dt \cdot |B_1(z)|^{-1}  \frac4{\lambda_1(M)}\frac{\sup|\chi'|^2}{R^2}\\
     & \le & C \lambda_1(M)^{-1}\, R^{-2}
     \end{eqnarray*}
     in view of (\ref{eq:L2-u_2}). Here $C_0$ is a universal constant and $C$ is a generic constant depending only on $b$, and $\inf_{x\in E} |B_1(x)|$.

    Put $\tilde{h}:=\chi(\rho/R)h-u$. Clearly, $\tilde{h}$ is a holomorphic differential on $M$ which satisfies
    $$
    \|\tilde{h}\|_2\le 1+ \frac2{\sqrt{\lambda_1(M)}} \sup|\chi'|\, R^{-1}
    $$
    and
    $$
    \left|\tilde{h}(z)\wedge \overline{\tilde{h}(z)}-h(z)\wedge \overline{h(z)}\right|=|u(z)|^2\le C\lambda_1(M)^{-1}\, R^{-2}
    $$
    for all $z\in E$. It follows that
    $$
    \sqrt{-1} K_M(z)\ge \sqrt{-1}\frac{\tilde{h}(z)\wedge \overline{\tilde{h}(z)}}{\|\tilde{h}\|^2_2}\ge \frac{\sqrt{-1}h(z)\wedge \overline{h(z)}-C\lambda_1(M)^{-1}\,R^{-2} \omega}{(1+C\lambda_1(M)^{-1/2}\,R^{-1})^2}
    $$
    where $\omega$ denotes the\/ {\it K\"ahler form\/} of $ds^2$. Since $h$ can be arbitrarily chosen, we have
    \begin{equation}\label{eq:BergmanLower}
    \sqrt{-1} K_M(z)-\frac{\sqrt{-1} K_{B_R}(z)}{(1+C\lambda_1(M)^{-1/2}\,R^{-1})^2}\ge -C\lambda_1(M)^{-1}\, R^{-2} \omega.
    \end{equation}

    Let $z_0\in E$ be fixed for a moment. We may choose a holomorphic differential $h_0$ on $M$ with unit $L^2-$norm, such that
    $$
    K_{B_R}(z_0)=h_0(z_0)\wedge \overline{h_0(z_0)}.
    $$
     Applying (\ref{eq:meanvalue}) to $|h_0|^2$ in a similar way as above, we obtain
    $
     |K_{B_R}(z_0)|\le C.
    $
    Together with (\ref{eq:BergmanLower}), we get
    $$
     \sqrt{-1} K_M(z_0)- \sqrt{-1} K_{B_R}(z_0)\ge -C(\lambda_1(M)^{-1/2}\,R^{-1}+\lambda_1(M)^{-1}\,R^{-2}) \omega.
    $$
      Since $\sqrt{-1} K_M(z)\le \sqrt{-1} K_{B_R}(z)$ holds trivially, so we conclude the proof.
    \end{proof}

    Note that $\lambda_1(M)=0$ holds for any compact Riemannian manifold. Thus we have to adjust the definition of $\lambda_1(M)$ as follows
    $$
    \lambda_1(M)=\inf \frac{\int_M |d f|^2}{\int_M |f|^2}
    $$
     where the infimum is taken over all $C^\infty-$smooth real-valued functions $f$ on $M$ such that $\int_M f=0$. Similar as (\ref{eq:eigenvalue}), we can verify that
     $$
      \lambda_1(M)=4 \inf \frac{\int_M |\partial f|^2}{\int_M |f|^2}
     $$
      where the infimum is taken over all $C^\infty-$smooth complex-valued functions $f$ on $M$ such that $\int_M f=0$.

      \begin{proposition}\label{lm:globalApprox2}
    Let $M$ be a\/ {\rm compact\/} Riemannian surface with $\lambda_1(M)>0$. Let $\rho$ denote the distance from some fixed point $o\in M$. Let $E$ be a compact subset in $M$ with $o\in E$. We have
    \begin{equation}\label{eq:golbalApprox2}
    \sup_{z\in E} \left|K_M(z)-K_{B_R}(z)\right|\le {\rm const.} \left(\lambda_1(M)^{-1/2}R^{-1}+\lambda_1(M)^{-1}R^{-2}\right)
    \end{equation}
    for all $R>2({\rm diam\,}E+1)$, where the constant depends only on $\inf_{x\in E} |B_1(x)|$ and the infimum of the Gaussian curvature of $ds^2$ on $E_1:=\{z\in M:{\rm dist\,}(z,E)\le 1\}$.
    \end{proposition}

    \begin{proof}
    Let $H^2_{(p,q)}(M)$ denote the space of $L^2$ harmonic $(p,q)-$forms on $M$. Clearly, $H^2_{(1,0)}(M)$ coincides with the Bergman space ${\mathcal H}(M)$. Now we determine $H^2_{(1,1)}(M)$. Since
    $$
    \bar{\partial}^\ast f=\partial \tilde{f}
    $$
    where $f=\phi dw\wedge d\bar{w}\in C^\infty_{(1,1)}(M)$ and $\tilde{f}=\lambda^{-1}\phi\in C^\infty(M,{\mathbb C})$, we see that $f_0:=\lambda dw\wedge d\bar{w}\in H^2_{(1,1)}(M)$, for $\bar{\partial}^\ast f_0=\bar{\partial} f_0=0$. On the other hand, it follows from the Serre Duality Theorem that $H^2_{(1,1)}(M)\backsimeq H^2_{(0,0)}(M)\backsimeq {\mathbb C}$. Thus $H^2_{(1,1)}(M)={\mathbb C}\cdot f_0$. Now suppose $f=\phi dw\wedge d\bar{w}\in H^2_{(1,1)}(M)^\bot\cap C^\infty_{(1,1)}(M)$, i.e., $(f,f_0)=\int_M  {\phi} dV_w=0$. It follows that $\tilde{f}=\lambda^{-1}\phi\in C^\infty(M,{\mathbb C})$ and $\int_M \tilde{f}=0$, so that
    $$
    \|\tilde{f}\|_2^2 \le \frac 4{\lambda_1(M)} \|\partial \tilde{f}\|_2^2,
    $$
    i.e.,
    $$
    \|f\|_2^2 \le \frac 4{\lambda_1(M)} \|\bar{\partial}^\ast f\|_2^2.
    $$
    The Hodge Decomposition Theorem asserts that
    $$
    C^\infty_{(1,0)}(M)={\mathcal H}(M)\oplus \bar{\partial}^\ast C^\infty_{(1,1)}(M).
    $$
    Thus for any $u\in {\mathcal H}(M)^\bot \cap C^\infty_{(1,0)}(M)$ we may write $u=\bar{\partial}^\ast f$ for some $f\in C^\infty_{(1,1)}(M)$. Without loss of generality, we may choose $f$ such that it is orthogonal to ${\rm Ker\,}\bar{\partial}^\ast=H^2_{(1,1)}(M)$. It follows that
    $$
    (u,u)=(u,\bar{\partial}^\ast f)=(\bar{\partial} u, f )\le \|\bar{\partial} u\|_2 \|f\|_2\le \frac2{\sqrt{\lambda_1(M)}}\|\bar{\partial} u\|_2 \|u\|_2,
    $$
    i.e.,
    $$
    \|u\|^2_2 \le \frac4{\lambda_1(M)} \|\bar{\partial} u\|^2_2.
    $$
   Let $\rho,\chi,h$ be as above. Replacing $\rho$ by a smoothing of it, we may assume, without loss of generality, that it is $C^\infty$ on $M$ and $|d\rho|\le 2$. Then we have the following orthogonal decomposition
   $$
   \chi(\rho/R)h=\tilde{h}\oplus u
   $$
    where $\tilde{h}\in {\mathcal H}(M)$ and $u\in {\mathcal H}(M)^\bot\cap C^\infty_{(1,0)}(M)$. Clearly, we have
    $$
    \bar{\partial}u=\bar{\partial}\chi(\rho/R)\wedge h=:v
    $$
    and
    $$
    \|u\|^2_2 \le \frac4{\lambda_1(M)} \|v\|^2_2.
    $$
    The remaining argument is essentially similar as Proposition \ref{lm:globalApprox}.
       \end{proof}

     \begin{corollary}
      Let $M$ be a complete Riemannian surface.
      \begin{enumerate}
      \item If $I_\infty(M)>0$ where $M$ can be compact or noncompact, then
    \begin{equation}\label{eq:golbalApproxIsoper}
    \sup_{z\in E} \left|K_M(z)-K_{B_R}(z)\right|\le C \left(I_\infty(M)^{-1}R^{-1}+I_\infty(M)^{-2}R^{-2}\right).
    \end{equation}
      \item If $M$ is compact, then
      \begin{equation}\label{eq:golbalApproxIsoper2}
    \sup_{z\in E} \left|K_M(z)-K_{B_R}(z)\right|\le C \left(\frac{\sqrt{|M|}}{I_2(M)R}+\frac{|M|}{I_2(M)^{2}R^{2}}\right).
    \end{equation}
      \end{enumerate}
      Here the constants are the same as the propositions above.
    \end{corollary}

    \begin{proof}
     Inequality (\ref{eq:golbalApproxIsoper}) follows from Proposition \ref{lm:globalApprox}, Proposition \ref{lm:globalApprox2},  and Cheeger's inequality
   $$
   \lambda_1(M)\ge \frac14 I_\infty(M)^2
   $$
   (see \cite{CheegerEigenvalue}).
   Inequality (\ref{eq:golbalApproxIsoper2}) follows from  Proposition \ref{lm:globalApprox2} and P. Li's estimate
   $$
   \lambda_1(M)\ge C_0\, \frac{I_2(M)^2}{|M|},
   $$
   where $C_0$ is a universal constant (see \cite{LiSobolev}, Proposition 3).
    \end{proof}

        \begin{proposition}\label{lm:globalApprox3}
    Let $M$ be a noncompact complete Riemannian surface with $I_\nu(M)>0$ for some $2<\nu<\infty$. Let $o\in M$  be fixed and let $B_R=B_R(o)$. For any compact subset $E\subset M$ with $o\in E$, we have
    \begin{equation}\label{eq:golbalApprox2}
    \sup_{z\in E} \left|K_M(z)-K_{B_R}(z)\right|\le C_1 |\log R|^{-1}
    \end{equation}
    for all $R>C_2$, where the constants $C_1,C_2$ depend only on $\nu$, $I_\nu(M)$, ${\rm diam\,}E$, $\inf_{x\in E} |B_1(x)|$, $|B_{2+{\rm diam\,}E}|$, $\inf_{x\in B_{2+{\rm diam\,}E}}{\rm inj}(M,x)$ and the infimum of the Gaussian curvature of $ds^2$ on $B_{2+{\rm diam\,}E}$.
    \end{proposition}

    \begin{proof}
   Let $\rho,\chi,h$ be as above. Put
   $$
   \kappa=\chi\left(-\log \log (g_M(\cdot,o)+1)+\log\log(C_\nu R^{2-\nu}+1)+1\right)
   $$
   where $C_\nu={\rm const}_\nu\,I_\nu(M)^\nu$ is the constant from (\ref{eq:GreenEstimate}). Note that
    $$
    {\rm supp\,}\kappa\subset \{g_M(\cdot,o)\ge C_\nu R^{2-\nu}\}\subset B_R
    $$
    in view of (\ref{eq:GreenEstimate}).
    Since
    $$
    -i\partial\bar{\partial}\log g_M(\cdot,o)\ge i \partial\log g_M(\cdot,o)\wedge \bar{\partial}\log g_M(\cdot,o)
    $$
    holds on $M\backslash \{o\}$, we infer from  the $L^2$ estimate of Donnelly-Fefferman (see \cite{DonnellyFefferman}, \cite{DiederichOhsawa}, \cite{Berndtsson96}) that there exists a solution of the equation
   $$
   \bar{\partial}u=\bar{\partial}(\kappa\,h)
   $$
  such that
   \begin{eqnarray*}
   \int_{M\backslash \{o\}} |u|^2 & \le & C_0 \int_{M\backslash \{o\}} |\bar{\partial}\kappa|^2_{-i\partial\bar{\partial}\log g_M(\cdot,o)} |h|^2\\
   & \le & {\rm const}_{\nu,I_\nu(M)} |\log R|^{-2}
   \end{eqnarray*}
   where $C_0$ is a universal constant.
   On the other side, since $g_M(\cdot,o)\ge C_3>0$ on $ B_{1+{\rm diam\,}E}$ in view of (\ref{eq:GreenLowerBound}), so we have
   $$
   \{\kappa=1\}\supset \left\{ g_M(\cdot,o)\ge (C_\nu R^{2-\nu}+1)^{e^{1/2}}-1 \right\}\supset B_{1+{\rm diam\,}E}
   $$
   provided $R>C_4$, where $C_3,C_4$ depend only on $\nu,I_\nu(M)$, ${\rm diam\,}E$, $\inf_{x\in B_{2+{\rm diam\,}E}}{\rm inj}(M,x)$,  $|B_{2+{\rm diam\,}E}|$  and the infimum of the Gaussian curvature of $ds^2$ on $B_{2+{\rm diam\,}E}$. Thus $u$ is holomorphic on $B_{1+{\rm diam\,}E}$ (note that $o$ is a removable singularity) and the remaining argument is similar as above.

    \end{proof}

 \section{Convergence of Riemannian surfaces and convergence of complex structures}
 Let $M$ be an orientable surface.  Let $\{ds^2_j\}$ be a sequence of Riemannian metrics on $M$, which converges to a Riemannian metric $ds^2$ on $M$ in the following sense: for each compact subset $E\subset M$ the tensor $ds^2_j- ds^2$ and its covariant derivatives of all orders (with respect to $ds^2$)  converge uniformly to zero on $E$.

    \begin{proposition}\label{prop:ComplexStructureConvergence}
    There exist a locally finite cover $\{U_\alpha\}$ of $M$ and holomorphic coordinates $w_j^{(\alpha)}$ (resp. $w^{(\alpha)}$) with respect to $ds_j^2$ (resp. $ds^2$) on $U_\alpha$ such that $w^\alpha_j-w^\alpha$ and its covariant derivatives of all order (with respect to $ds^2$)  converge uniformly to zero on $E\cap U_\alpha$ for any compact subset $E\subset M$.
\end{proposition}

  We believe that this result is essentially known (compare \cite{AhlforsBers}). For the sake of completeness, we will give a proof which relies upon the theory of elliptic operators of second order.  The key ingredient is the Lax-Milgram theorem which we recall as follows. A bilinear form ${\rm B}$ on a Hilbert space ${ H}$ is called\/ {\it bounded\/} if
$$
|{\rm B}(u,v)|\le {\rm const.} \|u\|\,\|v\|,\ \ \ {\forall\, }u,v\in { H}
$$
and\/ {\it coercive\/} if
$$
{\rm B}(u,u)\ge {\rm const.}\|u\|^2,\ \ \ {\forall\, }u\in { H}.
$$

\begin{theorem}[Lax-Milgram, cf. \cite{GilbargTrudinger}, Theorem 5.8]\label{th:LaxMilgram}
Let\/ ${\rm B}$ be a bounded, coercive bilinear form on a Hilbert space ${ H}$. For any bounded linear functional $\Phi$ on ${ H}$, there exists a unique element $v=v_\Phi\in { H}$ such that
$$
{\rm B}(u,v)=\Phi(u),\ \ \ \forall\, u\in { H}.
$$
\end{theorem}

  We begin with the classical Korn-Lichtenstein procedure (see, e.g., \cite{JostRiemann}, \S\,3.11).  Let $ds^2_j$ be given in local coordinates by
   $$
ds^2_j=E_j(x,y)dx^2+2F_j(x,y)dxdy+G_j(x,y)dy^2
$$
where $E_jG_j-F^2_j>0,\ E_j>0$. For abuse of notations, we denote $ds^2$ by $ds^2_\infty$. By introducing complex coordinates $z=x+iy,\bar{z}=x-iy$, we can write $ds^2_j$ in the form
$$
\sigma_j(z)|dz+\mu_j(z)d\bar{z}|^2
$$
where
$$
\sigma_j=\frac{E_j+G_j}4+\frac12 \sqrt{E_jG_j-F^2_j},\ \ \ \mu_j=\frac{E_j-G_j}{4\sigma_j}+i\frac{F_j}{2\sigma_j}.
$$
If there is a smooth solution $w_j$ of the following Beltrami equation
\begin{equation}\label{eq:Beltrami}
\frac{\partial w}{\partial\bar{z}}=\mu_j \frac{\partial w}{\partial z}
\end{equation}
such that
$$
\frac{\partial v_j}{\partial x}\frac{\partial u_j}{\partial y}-\frac{\partial u_j}{\partial x}\frac{\partial v_j}{\partial y}\neq 0
$$
 where $w_j=v_j+iu_j$, then the metric has the form
\begin{equation}\label{eq:coefficient}
ds^2_j=\frac{\sigma_j}{|\partial w_j/\partial z|^2} dw_j\,d\bar{w}_j,
\end{equation}
so that if $(U_\alpha,w^{(\alpha)}_j)$ and $(U_\beta,w^{(\beta)}_j)$ are two coordinate patches with $U_\alpha\cap U_\beta\neq \emptyset$, then
$$
\partial w^{(\beta)}_j/\partial \overline{w^{(\alpha)}_j}=0\ \ \ {\rm on\ }U_\alpha\cap U_\beta,
$$
i.e., $M$ admits a complex structure given by $\{(U_\alpha,w^{\alpha}_j)\}$, where $\{U_\alpha\}$ is a suitable cover of $M$. The point is that we can make the cover independent of $j$, in view of the convergence of $ds^2_j$.

With $w_j=v_j+iu_j$, (\ref{eq:Beltrami}) becomes
\begin{eqnarray}\label{eq:Beltrami_2}
\frac{\partial v_j}{\partial x} & = & -\frac{F_j}{\sqrt{E_jG_j-F^2_j}} \frac{\partial u_j}{\partial x}+\frac{E_j}{\sqrt{E_jG_j-F^2_j}} \frac{\partial u_j}{\partial y}\nonumber\\
\frac{\partial v_j}{\partial y} & = & -\frac{G_j}{\sqrt{E_jG_j-F^2_j}} \frac{\partial u_j}{\partial x}+\frac{F_j}{\sqrt{E_jG_j-F^2_j}} \frac{\partial u_j}{\partial y}
\end{eqnarray}
By using $\partial^2 v_j/\partial x\partial y=\partial^2 v_j/\partial y\partial x$, we derive that $u_j$ satisfies the following equation
\begin{equation}\label{eq:ellipticEq}
L_j u:=a_{11}^j \frac{\partial^2 u}{\partial x^2}-2a_{12}^j \frac{\partial^2 u}{\partial x \partial y}+a_{22}^j \frac{\partial^2 u}{\partial y^2}+b_1^j \frac{\partial u}{\partial x}+b_2^j \frac{\partial u}{\partial y}=0
\end{equation}
where
$$
a_{11}^j=\frac{G_j}{\sqrt{E_jG_j-F^2_j}},\ a_{12}^j=\frac{F_j}{\sqrt{E_jG_j-F^2_j}},\ a_{22}^j=\frac{E_j}{\sqrt{E_jG_j-F^2_j}},
$$
and
$$
b_1^j=\frac{a_{11}^j}{\partial x}-\frac{a_{12}^j}{\partial y},\ b_2^j=\frac{a_{22}^j}{\partial y}-\frac{a_{12}^j}{\partial x}.
$$

To solve (\ref{eq:Beltrami}), it suffices to find a smooth solution $u_j$ of the second-order\/ {\it elliptic\/} differential equation (\ref{eq:ellipticEq}) in some (simply-connected) neighborhood of an arbitrary point $z_0$ whose derivatives of order one at $z_0$ do not vanish, for $v_j$ may be determined by (\ref{eq:Beltrami_2}) such that
$$
\frac{\partial v_j}{\partial x}\frac{\partial u_j}{\partial y}-\frac{\partial u_j}{\partial x}\frac{\partial v_j}{\partial y}=\frac{1}{{E_jG_j-F^2_j}}\left( G_j \left(\frac{\partial{u_j}}{\partial x}\right)^2-2F_j \frac{\partial{u_j}}{\partial x}\frac{\partial{u_j}}{\partial y}+E_j \left(\frac{\partial{u_j}}{\partial y}\right)^2\right)>0
$$
holds at $z_0$.

 By an affine change of coordinates (which is uniform in $j$), we may assume that $z_0=0$ and
$$
a_{11}^j(0)=a_{22}^j(0)=1,\ \ \ a_{12}^j(0)=0.
$$
 Let $\Delta_r$ denote the disc in ${\mathbb R}^2$ with center $0$ and radius $r$. Fix $r>0$ for a moment. Let $C^\infty_0(\Delta_r)$ denote the set of real-valued smooth functions with compact supports in $\Delta_r$ and  $L^2(\Delta_r)$ be the completion of $C^\infty_0(\Delta_r)$ with respect to the $L^2-$norm $\|\cdot\|_2$ in the Lebesgue measure of ${\mathbb R}^2$. We may also define Sobolev spaces $W^{k,2}$ and $W^{k,2}_0$ by standard ways.

 Recall that the first eigenvalue $\lambda_1(\Delta_r)$ with respect to $-\partial^2/\partial x^2-\partial^2/\partial y^2$ equals $c\,r^{-2}$ for some numerical constant $c>0$, so that
$$
\|\nabla u\|^2_2 \ge c r^{-2}\|u\|^2_2,\ \ \ \forall\, u\in W^{1,2}_0(\Delta_r)
$$
 We introduce a bilinear form ${\rm B}_j$ for the Hilbert space $W^{1,2}_0(\Delta_r)$ as follows
\begin{eqnarray*}
{\rm B}_j(u,v) & = & (a_{11}^j u_x,v_x)-2(a_{12}^j u_x,v_y)+(a_{22}^j u_y,v_y)\\
& & -((b_1^j-\partial a_{11}^j/\partial x+2\partial a_{12}^j/\partial y)u_x,v)-((b_2^j-\partial a_{22}^j/\partial y)u_y,v).
\end{eqnarray*}
Clearly, we have
 $$
(-L_ju,u)  = ({\rm B}_ju,u),\ \ \ \forall\, u\in C^\infty_0(\Delta_r)
$$
and
$$
|{\rm B}_j(u,v)|\le C\|u\|_{{1,2}}\,\|v\|_{1,2},\ \ \ {\forall\, }u,v\in W^{1,2}_0(\Delta_r).
$$
Here and in what follows in this section we always assume that $j$ is\/ {\it sufficiently large\/} and  $C$ is a generic constant independent of $r$ and $j$.

Furthermore, we have
\begin{eqnarray*}
{\rm B}_j(u,u)  \ge  \frac34\, \|\nabla u\|^2_2-C\|u\|_2\,\|\nabla u\|_2
& \ge & \frac12\, \|u\|^2_{1,2}
\end{eqnarray*}
for all $u\in W^{1,2}_0(\Delta_r)$ and all $r\le \varepsilon_0$, where $\varepsilon_0$ is independent of $j$.

We look for a smooth function $\zeta_{j}$ satisfying
\begin{enumerate}
\item $\partial \zeta_j/\partial x=\partial \zeta_j/\partial y=1$ at $0$;
\item $L_j\zeta_{j}$ and its derivatives of order $\le 2$ vanish at $0$, i.e.,
\begin{equation}\label{eq:pointwise}
|L_j\zeta_j|\le Cr^3\ \ \ {\rm and\ \ \ } |\nabla (L_j\zeta_j)|\le Cr^2.
\end{equation}
\end{enumerate}

Let
$$
\xi_j=x+y-\frac{b_1^j(0)}2 x^2-\frac{b_2^j(0)}2y^2.
$$
Clearly, we have $L_j\xi_j(0)=0$. Put
$$
\eta_j=\xi_j-\frac16\, \frac{\partial L_j\xi_j}{\partial x}(0)\,{x^3}- \frac16\,\frac{\partial L_j\xi_j}{\partial y}(0)\,{y^3}.
$$
It is easy to see that $L_j\eta_j$ and its derivatives of order one vanish at $0$. Thus we may take
\begin{eqnarray*}
\zeta_j & = & \eta_j-\frac1{24}\,\frac{\partial^2 L_j\eta_j}{\partial x^2}(0)\,x^4-\frac1{24}\,\frac{\partial^2 L_j\eta_j}{\partial y^2}(0)\,y^4\\
&& -\frac1{12}\,\frac{\partial^2 L_j\eta_j}{\partial x \partial y}(0)(x^3y+xy^3).
\end{eqnarray*}

 By virtue of Theorem \ref{th:LaxMilgram}, we may find a solution $\hat{u}_j\in W^{1,2}_0(\Delta_r)$ of the equation
$$
-L_j u=L_j\zeta_j,
$$
which is smooth, for $L_j$ is elliptic and $L_j\zeta_j$ is smooth, and satisfies
$$
\|\hat{u}_j\|^2_{{1,2}}\le 2|{\rm B}_j(\hat{u}_j,\hat{u}_j)|=2|(\hat{u}_j,L_j\zeta_j)|\le 2\|\hat{u}_j\|_{{1,2}}\,\|L_j\zeta_j\|_2,
$$
i.e,
$$
\|\hat{u}_j\|_{{1,2}}\le \sqrt{2}\|L_j\zeta_j\|_2.
$$
By using dilatation $z\mapsto z/r$, we infer from
Sobolev's inequality and Garding's inequality (see e.g., \cite{GilbargTrudinger}, Theorem 8.10) that
\begin{eqnarray*}
\sup_{\Delta_{r/4}}|D\hat{u}_j| & \le &  C r^{-2}(\|L_j \hat{u}_j\|_{1,2}+\|\hat{u}_j\|_{1,2})\\
& \le & C r^{-2}\|L_j\zeta_j\|_{1,2}
\le  C r
\end{eqnarray*}
in view of (\ref{eq:pointwise}). Thus $u_j:=\zeta_j+\hat{u}_j$ gives a solution of the equation (\ref{eq:ellipticEq}), whose differential at every point of $\Delta_{r/4}$ does not vanish provided $r\le \varepsilon_0\ll 1$. The same is true for the corresponding isothermal parameter $w_j$.

Finally, we will verify the convergence of $\{w_j\}$. The argument is standard (see e.g., \cite{KodairaBook}, Theorem 7.5).  Fix first $r\le \varepsilon_0$. We have
$$
\|\hat{u}_j-\hat{u}\|_{{1,2}}\le \sqrt{2} \|L_j(\hat{u}_j-\hat{u})\|_2\le \sqrt{2}(\|L_j \zeta_j-L\zeta\|_2+\|(L_j-L) \hat{u}\|_2)\rightarrow 0
$$
as $j\rightarrow \infty$. It follows again from Sobolev's inequality and Garding's inequality that for each $l\in \{0\}\cup {\mathbb Z}^+$,
\begin{eqnarray*}
 && \sup_{\Delta_{r/4}} |D^l (\hat{u}_j-\hat{u})|  \le  {\rm const}_{l,r}\, \|\hat{u}_j-\hat{u}\|_{W^{l+2,2}(\Delta_{r/2})}\\
 & \le & {\rm const}_{l,r} (\|L_j(\hat{u}_j-\hat{u})\|_{l,2}+\|\hat{u}_j-\hat{u}\|_2)\\
  & \le & {\rm const}_{l,r} (\|L_j\zeta_j-L\zeta\|_{l,2}+\|(L_j-L)\hat{u}\|_{l,2}+\|\hat{u}_j-\hat{u}\|_2)
\end{eqnarray*}
where we use $D^l$ to denote any derivative of order $l$. Thus $D^l (\hat{u}_j-\hat{u})$ converges uniformly to zero on $\Delta_{r/4}$. The same is true for ${u}_j-{u}$ and $v_j-v$, hence for $w_j-w$.

  \section{Local stability of the Bergman kernel}

  Let $\Omega\subset\subset\Omega'\subset\subset M$ be two open sets. Let $\{ds^2_j\}$ be a sequence of Riemannian metrics on $\Omega'$, which converges uniformly on $\overline{\Omega}$ to a Riemannian metric $ds^2$ on $\Omega'$ in the following sense:  the tensor $ds^2_j- ds^2$ and its covariant derivatives of all orders (with respect to $ds^2$) converge uniformly to zero on $\overline{\Omega}$.
  By virtue of Proposition \ref{prop:ComplexStructureConvergence}, we can choose a locally finite cover $\{U_\alpha\}$ of $\Omega'$
   and holomorphic coordinates $w_j^{(\alpha)}$ (resp. $w^{(\alpha)}$) with respect to $ds_j^2$ (resp. $ds^2$) on $U_\alpha$ such that $w^\alpha_j-w^\alpha$ and its covariant derivatives of all order (with respect to $ds^2$)  converge uniformly to zero on $\overline{\Omega}$.

     Let $K_{D,j}$ (resp. $K_D$) denote the Bergman kernel of an open set $D\subset \Omega'$, with respect to the complex structure $\{(U_\alpha,w^{(\alpha)}_j)\}$ (resp. $\{(U_\alpha,w^{(\alpha)})\}$).

   \begin{proposition}\label{lm:localApprox}
   For each $\varepsilon>0$ and each compact set $E\subset \Omega$, there exists an integer $j_0$ such that for all $j\ge j_0$ and $z\in E$ we have
   $$
   \sqrt{-1} K_{\Omega,j}(z)\ge \sqrt{-1} K_{\Omega'}(z)-\varepsilon \omega,\ \ \ \sqrt{-1} K_{\Omega}(z)\ge \sqrt{-1} K_{\Omega',j}(z)-\varepsilon \omega.
   $$
   Here $\omega$ denotes the K\"ahler form of $ds^2$.
   \end{proposition}

   \begin{proof}
   Fix $z_0\in E$ for a moment. Let $f$ be a holomorphic differential on the Riemann surface $(\Omega',w)$ which satisfies
   $$
   f(z_0)\wedge \overline{f(z_0)}=K_{\Omega'}(z_0),\ \ \ \frac{\sqrt{-1}}2 \int_{\Omega'} f\wedge \bar{f} =1.
   $$
   By Cauchy's estimates, we have
   \begin{equation}\label{eq:CauchyBound}
   \sup_{\overline{\Omega}} \{|f|^2, |{\partial} f|^2\}\le {\rm const}_{\Omega,\Omega'} \left|\int_{\Omega'} f\wedge \bar{f} \right|={\rm const}_{\Omega,\Omega'},
   \end{equation}
   where $|\cdot|$ is given with respect to $ds^2$.
   Since
   $$
   \frac{\partial f^\ast}{\partial \bar{w}_j}=\frac{\partial f^\ast}{\partial {w}} \frac{\partial w}{\partial \bar{w}_j}+ \frac{\partial f^\ast}{\partial \bar{w}}\frac{\partial \bar{w}}{\partial \bar{w}_j}=\frac{\partial f^\ast}{\partial {w}} \frac{\partial w}{\partial \bar{w}_j},
   $$
   where $f^\ast$ is a local representation of $f$,
   so we obtain
   \begin{equation}\label{eq:Cauchy-dbar}
   \sup_{\overline{\Omega}}|\bar{\partial}_j f|^2\le {\rm const}_{\Omega,\Omega'}\,\sup_{\overline{\Omega}}\left|\frac{\partial w}{\partial \bar{w}_j}\right|^2=:{\rm const}_{\Omega,\Omega'}\,\varepsilon_j
    \end{equation}
    where $\bar{\partial}_j$ denotes the Cauchy-Riemann operator for the Riemann surface $(\Omega',w_j)$.

       Since $(\Omega',w)$ is a Stein manifold in view of the Behnke-Stein theorem, it admits a smooth strictly subharmonic function $\varphi$. We claim that $\varphi$ is also strictly subharmonic on $(\overline{\Omega},w_j)$ provided $j$ sufficiently large. To see this, simply note that
   \begin{eqnarray}\label{eq:ddbar}
    \frac{\partial^2 \varphi}{\partial w_j \partial \bar{w}_j} & = & \left(\frac{\partial^2\varphi}{\partial w^2} \frac{\partial w}{\partial w_j}+ \frac{\partial^2\varphi}{\partial w \partial\bar{w}} \frac{\partial \bar{w}}{\partial w_j}   \right) \frac{\partial w}{\partial \bar{w}_j}+\frac{\partial \varphi}{\partial {w}} \frac{\partial^2 w}{\partial w_j\partial\bar{w}_j}\nonumber\\
    & & + \left(\frac{\partial^2\varphi}{\partial w \partial\bar{w}} \frac{\partial w}{\partial w_j}+ \frac{\partial^2\varphi}{\partial\bar{w}^2} \frac{\partial \bar{w}}{\partial w_j}   \right) \frac{\partial \bar{w}}{\partial \bar{w}_j}+\frac{\partial \varphi}{\partial \bar{w}} \frac{\partial^2 \bar{w}}{\partial w_j\partial\bar{w}_j}\nonumber\\
        & \rightarrow & \frac{\partial^2\varphi}{\partial w \partial\bar{w}}
   \end{eqnarray}
   uniformly on $\overline{\Omega}$ as $j\rightarrow \infty$.

    By H\"ormander's $L^2-$estimates for $\bar{\partial}$ (see e.g., \cite{HormanderConvexity}), there exists a smooth solution $u_j$ of $\bar{\partial}_j u=\bar{\partial}_j f$ on $\Omega$ such that
\begin{equation}\label{eq:L2-dbar_3}
\frac{\sqrt{-1}}2\int_{\Omega} u_j\wedge \bar{u_j}e^{-\varphi} \le \int_{\Omega} |\bar{\partial}_j f|^2_{i\partial_j\bar{\partial}_j\varphi} e^{-\varphi} dV_j
\le {\rm const}_{\Omega,\Omega',\varphi}\,\varepsilon_j
\end{equation}
 in view of (\ref{eq:Cauchy-dbar}) and (\ref{eq:ddbar}) (note also that $\varphi$ is bounded on $\Omega$). Put $f_j=f-u_j$. We see that $f_j$ is a holomorphic differential on the Riemann surface $(\Omega,w_j)$ such that
 $$
 \|f_j\|_2 \le \|f\|_2+\|u_j\|_2 \le 1+{\rm const}_{\Omega,\Omega',\varphi}\,\sqrt{\varepsilon_j}.
 $$
 Now fix a positive number $r=r(\Omega,\Omega')$ such that there is a holomorphic coordinate disc $\Delta_r(z_0)\subset (\Omega,w_j)$ for all sufficiently large $j$. Write $u_j=u^\ast_j dw_j$ on $\Delta_r(z_0)$. Let $\chi$ be the cut-off function in the proof of Proposition \ref{lm:globalApprox}. By the Cauchy integral formula, we have
 \begin{eqnarray*}
 u^\ast_j(z_0) & = & \frac1{2\pi \sqrt{-1}} \int \frac{\bar{\partial}(\chi(|w_j|/r_0)u^\ast_j)/\partial \bar{w}_j}{w_j} dw_j\wedge d\bar{w}_j\\
  & = & \frac1{2\pi \sqrt{-1}} \int \frac{u^\ast_j\bar{\partial}\chi(|w_j|/r_0)/\partial \bar{w}_j}{w_j} dw_j\wedge d\bar{w}_j\\
  && +\frac1{2\pi \sqrt{-1}} \int \frac{\chi(|w_j|/r_0)\bar{\partial}u^\ast_j/\partial \bar{w}_j}{w_j} dw_j\wedge d\bar{w}_j.
 \end{eqnarray*}
 Thus
 $$
 |u_j(z_0)|^2\le {\rm const}_{\Omega,\Omega',\varphi}\,\varepsilon_j
 $$
 in view of (\ref{eq:Cauchy-dbar}), (\ref{eq:L2-dbar_3}). Thus
  $$
 \sqrt{-1} K_{\Omega,j}(z_0)\ge \sqrt{-1} \frac{f_j(z_0)\wedge \overline{f_j(z_0)}}{\|f_j\|_2^2}\ge \sqrt{-1} K_{\Omega'}(z_0)-{\rm const}_{\Omega,\Omega',\varphi}\,\sqrt{\varepsilon_j}\,\omega,
 $$
 for $|K_{\Omega'}(z_0)|\le {\rm const}_{\Omega,\Omega'}$.
  The other inequality can be verified similarly.
      \end{proof}

\section{Proof of Theorem \ref{th:convergence}}

     Clearly the convergence $(M_j,ds^2_j,p_j)\rightarrow (M,ds^2,p)$ implies that
  \begin{equation}\label{eq:curvatureBound}
  \sup_j\sup_{B_r(p_j)}|\nabla^k {\rm Rm}(ds^2_j)|<\infty,\ \ \ \forall\,r>0,\ \forall\,k\in \{0\}\cup {\mathbb Z}^+
  \end{equation}
    where ${\rm Rm}$ denotes the Riemannian curvature tensor and $\nabla^k$ denotes any covariant derivative of order $k$.

   For the sake of simplicity, we put $B^j_r=B_r(p_j)$ and $B_r=B_r(p)$.
  Let $E\subset M$ be a compact set with $p\in E$. Put $E_j=\phi_j(E)$.
  By virtue of (\ref{eq:curvatureBound}),  (\ref{eq:golbalApproxIsoper}), (\ref{eq:golbalApproxIsoper2}) and (\ref{eq:golbalApprox2}), there exists a positive constant $R_0$ such that for all $R\ge R_0$,
   \begin{equation}\label{eq:Bergman-M_j}
        \sup_{z\in E_j} \left|K_{M_j}(z)-K_{B^j_R}(z)\right|_{ds^2_j}  \le  C |\log R|^{-1}.
    \end{equation}
    Here and in what follows in this section we always assume that $j$ is\/ {\it sufficiently large\/}, and $C$ is a generic constant independent of $j$ and $R$.

      Now fix an arbitrary number $0<\varepsilon<|\log R_0|^{-1}$ for a moment. Put $R=e^{1/\varepsilon}$. It is easy to see that
      $$
      B_{R-1}\subset \Omega_R^j:=\phi_j^{-1}(B^j_R)\subset B_{R+1}.
      $$
            By (\ref{eq:Bergman-M_j}), we have
      \begin{eqnarray}\label{eq:Bergman-M_j-2}
        \sup_{z\in E} \left|\phi_j^\ast K_{M_j}(z)-\phi_j^\ast K_{B^j_R}(z)\right|_{\phi_j^\ast(ds^2_j)}  \le  C \varepsilon.
      \end{eqnarray}
      The point is that $\phi_j^\ast K_{B^j_R}$ actually coincides with $K_{\Omega_R^j,j}$, the Bergman kernel of $\Omega_R^j$ with respect to the complex structure induced by $\phi_j^\ast(ds^2_j)$.  Thus (\ref{eq:Bergman-M_j-2}) yields
       \begin{eqnarray}\label{eq:BergmanLower-M_j}
         \sqrt{-1} \phi_j^\ast K_{M_j}(z)\ge \sqrt{-1} K_{\Omega_R^j,j}(z) -C  \varepsilon\,\omega,\ \ \ \forall\,z\in E.
      \end{eqnarray}
     Since $\phi^\ast_j(ds^2_j)\rightarrow ds^2$ uniformly on $B_{R+2}$ in the sense of \S\,7, we have
      \begin{eqnarray}\label{eq:BergmanLower-M_j-2}
   \sqrt{-1} K_{\Omega_R^j,j}(z) & \ge & \sqrt{-1} K_{B_{R+1},j}(z) \ge \sqrt{-1} K_{B_{R+2}}(z)-\varepsilon \omega\nonumber\\
    & \ge & \sqrt{-1} K_M(z)- \varepsilon\,\omega
   \end{eqnarray}
   in view of Proposition \ref{lm:localApprox}. Thus
      $$
   \sqrt{-1} \phi_j^\ast K_{M_j}(z)\ge \sqrt{-1} K_M(z)-C \varepsilon\,\omega,\ \ \ \forall\,z\in E
   $$
   in view of (\ref{eq:BergmanLower-M_j}) and (\ref{eq:BergmanLower-M_j-2}).
   It follows again from Proposition \ref{lm:localApprox} that
   $$
   \sqrt{-1} K_{\Omega_R^j,j}(z)  \le  \sqrt{-1} K_{B_{R-1},j}(z)\le \sqrt{-1} K_{B_{R-2}}(z)+\varepsilon\omega.
   $$
   Together with (\ref{eq:Bergman-M_j-2}), we obtain
   $$
   \sqrt{-1} \phi_j^\ast K_{M_j}(z)\le \sqrt{-1} K_{B_{R-2}}(z)+C\varepsilon \omega.
   $$
   A normal family argument shows
   $$
   K_{B_{R-2}}(z)\rightarrow K_M(z)
   $$
   uniformly on $E$ as $R\rightarrow \infty$. Thus we have verified that $\phi_j^\ast K_{M_j}(z)\rightarrow K_M(z)$.  The convergence of their covariant derivatives can be verified by using the Cauchy integral formula. We leave the details to the reader.

    \section{Proof of Theorem \ref{th:convergenceEffective}}

    Let $E$ be a compact set in $M$ which contains $p$ and $E_1:=\{z\in M:{\rm dist}(z,E)\le 1\}$.  Let $\{h_k\}$ be a complete orthonormal basis of ${\mathcal H}(M)$. Let $K_M(z,w)$ denote the off-diagonal Bergman kernel, i.e.,
    $$
    K_M(z,w)=\sum_k h_k(z)\otimes \overline{h_k(w)}.
    $$
      We define $|K_M(z,w)|$ as follows. Let $ds^2=\lambda dw d\bar{w}$ (resp. $\mu dz d\bar{z}$) at $w$ (resp. $z$). If we write $K_M(z,w)=K_M^\ast(z,w)dz\otimes d\bar{w}$, then
  $$
  |K_M(z,w)|^2:=\frac{|K_M^\ast(z,w)|^2}{\lambda(w)\mu(z)}.
  $$
    Fubini's theorem yields
    $$
    \int_{M\times E_1} |K_M(z,w)|^2=\int_{E_1} |K_M(w,w)|^2<\infty,
    $$
    so that for any $\varepsilon>0$,
    $$
    \int_{(M\backslash B_R(p))\times E_1 } |K_M(z,w)|^2<\varepsilon
    $$
    provided $R$ sufficiently large. Applying the sub-mean-value inequality as the proof of Proposition \ref{lm:globalApprox}, we obtain
    $$
    |K_M(z,w)|^2\le {\rm const.}\,  \int_{\zeta\in E_1 } |K_M(z,\zeta)|^2
    $$
    for all $z\in M$ and $w\in E$, where the constant depends only on $E$. Thus
        $$
    \int_{z\in M\backslash B_R(p)}|K_M(z,w)|^2\le {\rm const.}\,  \varepsilon.
    $$
         Put
    $$
    h(z,w)=\sum_k \overline{h^\ast_k(w)}\, h_k(z)
    $$
    where $h^\ast_k$ is a local representation of $h_k$ at $w$.  Then we have
    $$
    \lambda(w)^{-1}\int_{z\in M\backslash B_R(p)}|h(z,w)|^2\le {\rm const.}\,  \varepsilon.
    $$
    Now fix $w\in E$ and $R$ for a moment. Clearly, we have $R_j>2R$ for  $j\gg 1$. Let $\chi$ be the cut-off function in the proof of Proposition \ref{lm:globalApprox} and let $\rho$ denote the distance from $p$ on $M$. By the proofs of Propositions \ref{lm:globalApprox}, \ref{lm:globalApprox2}, we may write
    $$
   \chi(\rho/R_j)h(\cdot,w)=\tilde{h}_j\oplus u_j,
   $$
    where $\tilde{h}_j\in {\mathcal H}(M_j)$ and $u_j$ satisfies
    $$
    \bar{\partial}u_j=\bar{\partial}\chi(\rho/R_j)\wedge h(\cdot,w)=:v_j
    $$
    and
    \begin{eqnarray*}
     && \int_{M_j} |u_j|^2  \le  \frac4{\lambda_1(M_j)} \int_{M_j} |v_j|^2\\
     & \le &  \frac{C_0}{\lambda_1(M_j)R_j^2}\,\int_{M\backslash B_{\frac12R_j}(p)}|h(\cdot,w)|^2\\
     & \le & {\rm const.}\lambda(w)\,\varepsilon
    \end{eqnarray*}
    where $C_0$ is a universal constant.
    Again by the sub-mean-value inequality,
    $$
    |u_j(w)|^2 \le {\rm const.}\lambda(w)\,\varepsilon,
    $$
    i.e., $|u_j^\ast(w)|^2\le {\rm const.}\lambda(w)^2\,\varepsilon$. Note that
    \begin{eqnarray*}
    \|\tilde{h}_j\|_{L^2(M_j)} & \le &  \|h(\cdot,w)\|_{L^2(M)}+\|u_j\|_{L^2(M_j)}\\
    & \le & K_{M}^\ast(w,w)^{1/2}+O(\sqrt{\lambda(w)\,{\varepsilon}}).
    \end{eqnarray*}
        It follows that
    $$
    K_{M_j}^\ast(w,w)\ge {\rm } \frac{K_{M}^\ast(w,w)^2-O(\lambda(w)^2\,\varepsilon)}{(K_{M}^\ast(w,w)^{1/2}+O(\sqrt{\lambda(w)\,{\varepsilon}}))^2},
    $$
    i.e.,
    $$
    |K_{M_j}(w,w)|\ge \frac{|K_{M}(w,w)|^2-O(\varepsilon)}{(|K_{M}(w,w)|^{1/2}+O(\sqrt{\varepsilon}))^2}\ge |K_{M}(w,w)|-O(\sqrt{\varepsilon}).
    $$
    On the other side, we have
    $$
     |K_{M_j}(w,w)|\le |K_{B_{R_j}(p)}(w,w)|\rightarrow |K_{M}(w,w)|
    $$
    as $j\rightarrow \infty$. The proof is complete.

\section{Concluding remarks}

1. In general, the Cheeger-Gromov convergence does not imply the convergence of the Bergman kernel.

     \begin{proposition}\label{prop:Infnite-genus}
     Let $\{(M_j,ds^2_j)\}$ be a sequence of\/ {\it compact\/} Riemannian surfaces which converges to a complete Riemannian surface $(M,ds^2)$ such that
      \begin{enumerate}
       \item ${\rm dim\,}{\mathcal H}(M)=\infty$,
       \item $K_{M_j}\rightarrow K_M$ in the sense of Theorem \ref{th:convergence}.
       \end{enumerate}
       Then the genus $g_j$ of $M_j$ tends to infinity as $j\rightarrow \infty$.
      \end{proposition}

      \begin{proof}
      Take first a complete orthonormal basis $h_{j,1},\cdots,h_{j,g_j}$ of ${\mathcal H}(M_j)$.  Note that
  $$
  K_{M_j}(z)=\sum_{k=1}^{g_j} h_{j,k}(z)\wedge \overline{ h_{j,k}(z)}.
  $$
      Thus
  \begin{eqnarray*}
  g_j ={\rm dim}_{\mathbb C}({\mathcal H}(M_j))& = & \frac{\sqrt{-1}}2 \int_{M_j}K_{M_j}(z) \ge  \frac{\sqrt{-1}}2 \int_{\phi_j(\Omega_j)}K_{M_j}(z) \\
                                               & = & \frac{\sqrt{-1}}2 \int_{\Omega_j} \phi_j^\ast K_{M_j}
  \end{eqnarray*}
  where $\phi_j$ and $\Omega_j$ are given as Definition 1.1.
  Since $\phi^\ast K_{M_j}\rightarrow K_M$ locally uniformly on $M$, so we have
  $$
  \lim\inf_{j\rightarrow \infty} g_j \ge \frac{\sqrt{-1}}2\int_E K_M(z)
  $$
  for any compact set $E\subset M$. It follows immediately that
  $$
  \lim\inf_{j\rightarrow \infty} g_j \ge \frac{\sqrt{-1}}2\int_M K_M(z)=\infty.
  $$
  \end{proof}

  \begin{example}
    Let $M$ be a semi-sphere in the sphere ${S}^2 \subset {\mathbb R}^3$, which is conformally equivalent to the Poincar\'e disc $(\Delta,ds^2)$ by Riemann's mapping theorem. Let $\{\Omega_j\}$ be a sequence of precompact open subsets exhausting $M$.  Let $M_j=({S}^2,ds_j^2)$ where $ds^2_j$ is a Riemannian metric with $ds^2_j=ds^2$ on $\Omega_j$.
    Thus $(M_j,ds^2_j)\rightarrow (M,ds^2)$, whereas $K_{M_j}\nrightarrow K_M$.
  \end{example}

    2. The case of convergent hyperbolic surfaces is of independent interest. Let us recall the following

    \begin{definition}[cf. \cite{Thurston97}]
     A sequence $\{\Gamma_j\}$ of closed subgroups of a Lie group converges geometrically to a group $\Gamma$ if
     \begin{enumerate}
     \item each $\gamma\in \Gamma$ is the limit of a sequence $\{\gamma_j\}$, with $\gamma_j\in \Gamma_j$,
     \item the limit of every convergent sequence $\{\gamma_{j_k}\}$, with $\gamma_{j_k}\in \Gamma_{j_k}$ is in $\Gamma$.
     \end{enumerate}
    \end{definition}

It is known that if a sequence $\{\Gamma_j\}$ of torsion-free Fuchsian groups converges geometrically to a\/ {\it non-elementary\/} Fuchsian group $\Gamma$, then ${\mathbb D}/\Gamma_j\rightarrow {\mathbb D}/\Gamma$ in the sense of Cheeger-Gromov, where ${\mathbb D}$ denotes the unit disc (see e.g., \cite{MatsuzakiTaniguchi}, Theorem 7.6). Thus by Theorem \ref{th:convergence} and (\ref{eq:PoincareExp}), we immediate get the following

\begin{proposition}
Let $\{\Gamma_j\}$ be a sequence of torsion-free Fuchsian groups converges geometrically to a\/ {\it non-elementary\/} Fuchsian group $\Gamma$ and satisfies
$
\sup_j \delta(\Gamma_j)<1.
$
Then $K_{{\mathbb D}/\Gamma_j}\rightarrow K_{{\mathbb D}/\Gamma}$.
\end{proposition}

3. Following \cite{BurgerBuserDodziuk}, we may construct a family $\{M_t\}$ of compact Riemannian surfaces with $\lambda_1(M_t)\ge {\rm const.}>0$ in the following way:

Let $M$ be a hyperbolic Riemannian surface with $2n$ cusps $\{C_i\}$. Let $M_t$ be the surface formed from $M$ by first replacing the punctures w.r.t. $C_i$ with geodesics of length $t$ and then gluing the geodesic w.r.t. $C_{2i-1}$ with the geodesic w.r.t. $C_{2i}$. One may perturb the hyperbolic metric on $M$ slightly to obtain a new Riemannian metric on $M_t$, so that $M_t\rightarrow M$ in the sense of Cheeger-Gromov, and $\lambda_1(M_{t})\ge {\rm const.}>0$ as $t\rightarrow 0$.

Another interesting way is to consider a family of non-singular connected levels $M_t=f^{-1}(t)$, $t\in {\mathbb R}$, of a polynomial $f$ on the Euclidean sphere $S^3$, when $M_t$ approaches the singular surface $M_s$ (w.r.t. the Euclidean metric) as $t\rightarrow s$. It is known that  if $M_s$ is irreducible and ${\rm dim\,}V_s={\rm dim\,}V_t$ then $\lambda_1(M_t)\ge {\rm const.}>0$ as $t\rightarrow s$ (see Gromov \cite{GromovSpectral}, p. 252).

4. A\/ {\it conic\/} degenerating family $\{M_t\}$ of compact Riemannian surfaces has $I_2(M_t)\ge {\rm const.}>0$ and $|M_t|\le {\rm const.}$, provide that the pinching geodesic is nonseparating (see \cite{JiWentworth}, Corollary 2.9 and the proof of Proposition 2.6). It follows that
$$
I_2(M_t)|M_t|^{-1/2}\ge {\rm const.}>0.
$$
 Effective convergence the Bergman kernel and related invariants for some special degenerating\/ {\it analytic\/} families of compact Riemannian surfaces was established in \cite{Wentworth}, \cite{JiWentworth}.

5.   Let $M_\infty $ be a complete Riemannian manifold such that there exists a nested sequence of torsion-free discrete groups of isometries
$$
\Gamma_1\supset\Gamma_2\supset\cdots\supset \Gamma_j\supset\cdots\supset \cap \Gamma_j=\{\rm id\}.
$$
 Let  $M_j=M_\infty/\Gamma_j$  be endowed with the metric induced by the complete metric on $M_\infty$.  Then $M_j\rightarrow M_\infty$ in the sense of Cheeger-Gromov (we may choose $\phi_j=\pi_j|_{{\mathcal D}_j}$ where $\pi_j:M_\infty\rightarrow M_j$ is the covering map and ${\mathcal D}_j$ is suitable fundamental region). After the seminal work of Kazhdan (see \cite{Kazhdan1}, \cite{Kazhdan2}), stability properties of the Bergman kernel when all $M_j$ are complex manifolds were studied extensively (cf. \cite{Rhodes}, \cite{donnelly96}, \cite{Yeung00}, \cite{Ohsawa10}, \cite{ChenFuTower}, \cite{Wang}). In particular, Rhodes \cite{Rhodes} proved $K_{M_j}\rightarrow K_M$ for Riemannian surfaces satisfying $\lambda_1(M_j)\ge {\rm const.}>0$, whereas  Ohsawa \cite{Ohsawa10} gave a counterexample for general case.

6.   Below we provide a sequence of compact Riemannian surfaces which satisfies the conditions of Theorem \ref{th:convergenceEffective}, whereas Theorem \ref{th:convergence} does not apply.
  Let us start from a compact Riemann surface $M_0$ with genus $\ge 2$. Let $M$ be a regular covering of $M_0$ whose deck transformation group $\Gamma$ is isomorphic with ${\mathbb Z}$. For instance, we may choose $M$ to be a Schottkyan type covering of $M_0$ by first taking a ring cut $\gamma$ of $M_0$ then connecting infinitely many copies of $M_0\backslash \gamma$ along the opposite shores of $\gamma$ (see e.g. \cite{TsujiBook}, Chapter X, \S\,14).

Put $\Gamma=\{g_k:k\in {\mathbb Z}\}$. Let $M_j$ be a compact Riemannian surface obtained by adding a spherical cap ${\mathcal C}$ to each end of the set
$$
  \bigcup_{\{k\in {\mathbb Z}:|k|\le j\}} g_k(M_0\backslash \gamma).
  $$
    We may introduce a Riemannian metric on $M_j$ by patching up together the metric on $M$ and the Euclidean metric on the spherical cap. Clearly there exists $R_j\approx j$ such that $B_{R_j}(p)\subset M_j$ for some fixed point $p\in M$.
    Since $M$ has bounded geometry, so we have the following isoperimetric inequality
    \begin{equation}\label{eq:BoundLength}
    |\partial \Omega|\ge {\rm const.}\min\left\{1,\sqrt{|\Omega|}\right\}
    \end{equation}
    (see \cite{GrigoryanHeat}, Theorem 7.7).
    We claim that
    \begin{equation}\label{eq:IsoperEstimate}
    I_\infty(M_j)\ge {\rm const.}\,j^{-1}.
    \end{equation}
    To see this, let $S$ be a smooth hypersurface that divides $M_j$ into two disjoint open subsets $\Omega_1,\Omega_2$. According to Yau \cite{Yau75}, it suffices to consider the case when both $\Omega_1$ and $\Omega_2$ are connected.
    Suppose $|\Omega_1|\le |M_j|/2$.  If $\Omega_1\subset M$,  we infer from (\ref{eq:BoundLength}) that
    $$
    \frac{|\partial \Omega_1|}{|\Omega_1|}\ge {\rm const.} |\Omega_1|^{-1}\ge {\rm const.} |M_j|^{-1}\ge {\rm const.}\,j^{-1}
    $$
    when $|\Omega_1|\ge 1$, and
    \begin{equation}\label{IsopR2}
    \frac{|\partial \Omega_1|}{|\Omega_1|}\ge {\rm const.}
    \end{equation}
    when $|\Omega_1|\le 1$.
    If $\Omega_1$ is contained in a spherical cap slightly larger than ${\mathcal C}$, then we still have (\ref{IsopR2}) in view of the classical isoperimetric inequality in ${\mathbb R}^2$.
            In the remaining case, we put $\Omega_1'=\Omega_1\cap M$ and $\Omega_1''=\Omega_1\backslash \overline{\Omega_1'}$. Then we have
    $$
    |\partial \Omega_1|\ge {\rm const.}\max\{\partial \Omega_1',\partial \Omega_1''\},
    $$
    so that
    $$
    \frac{|\partial \Omega_1|}{|\Omega_1|}\ge {\rm const.}\min\left\{\frac{|\partial \Omega_1'|}{|\Omega_1'|},\frac{|\partial \Omega_1''|}{|\Omega_1''|}\right\}\ge {\rm const.}\,j^{-1}.
    $$
    Thus we have verified (\ref{eq:IsoperEstimate}). Finally, Cheeger's inequality yields
    $$
    \lambda_1(M_j)\ge {\rm const.}\,j^{-2}\approx R_j^{-2}.
    $$
    The same argument probably works when $M$ is a ${\mathbb Z}^m$ covering of $M_0$ for arbitrary $m\in {\mathbb Z}^+$.
    
7. We end this section by proposing the following

   \begin{problem}
  Let $\{f_j\}$ be a sequence of smooth functions in ${\mathbb R}^3$ which converges locally uniformly to a smooth function $f$. Suppose $M_j:=\{f_j=0\}$ and $M:=\{f=0\}$ are non-singular. With respect to the complex structure induced by the Euclidean metric, when does $K_{M_j}$ converge to $K_M$ in some sense?
 \end{problem}

 \bigskip

{\bf Acknowledgement.}
 The author would like to thank Professors Zbigniew Blocki and Slawomir Kolodziej for their valuable comments at Oberwolfach in January 2015.

\end{document}